\documentclass[12pt]{article}

\usepackage{amsmath, amssymb, latexsym, amsthm, setspace, color, fullpage}
\usepackage[numbers]{natbib}
\newtheorem{theorem}{Theorem}[section]
\newtheorem{lemma}[theorem]{Lemma}

\newtheorem{corollary}[theorem]{Corollary}
\newtheorem{conjecture}[theorem]{Conjecture}

\theoremstyle{definition}
\newtheorem{definition}[theorem]{Definition}

\newtheorem{remark}[theorem]{Remark}

\def\nul{\varnothing} 
\def\st{\colon\,}   
\def\MAP#1#2#3{#1\colon\,#2\to#3}

\def\SM#1#2{\sum_{#1\in #2}}
\def\FR#1#2{\frac{#1}{#2}}
\def\FL#1{\left\lfloor{#1}\right\rfloor}

\def\esub{\subseteq}
\def\E#1{\left\Vert#1\right\Vert}
\def\NN{{\mathbb N}}

\def\C#1{\left|#1\right|}
\def\Ccut#1{\bigl|#1\bigr|}

\def\Arb{{\rm Arb}}
\def\eps{\epsilon}

\def\nosub{\nsubseteq}
\def\ti{\hfil\break\hglue\the\parindent}

\begin{document}

\title{Decomposition of Sparse Graphs into Forests: The Nine Dragon Tree
Conjecture for $k\le2$}

\author{
Min Chen\thanks{Department of Mathematics, Zhejiang Normal University,
chenmin@zjnu.edu.cn.}\,,
Seog-Jin Kim\thanks{Department of Mathematics Education, Konkuk University,
skim12@konkuk.ac.kr.  Research supported by Basic Science Research Program
through the National Research Foundation of Korea (NRF), funded by the Ministry
of Education, No. 2011-0009729.}\,,
Alexandr Kostochka\thanks{Departments of Mathematics, University of Illinois
and Zhejiang Normal University, kostochk@math.uiuc.edu.  Research supported in
part by NSF grant DMS-1266016.}\,,
\\
Douglas B. West\thanks{Departments of Mathematics, Zhejiang Normal University
and University of Illinois, west@math.uiuc.edu.  Research supported by
Recruitment Program of Foreign Experts, 1000 Talent Plan, State Administration
of Foreign Experts Affairs, China.
}\,,
Xuding Zhu\thanks{Department of Mathematics, Zhejiang Normal University,
xdzhu@zjnu.edu.cn.  Research supported by NSF11171310 and ZJNSF Z6110786.}
}
\date{November 13, 2014}
\maketitle

\vspace{-2.5pc}
\begin{abstract}
For a loopless multigraph $G$, the {\it fractional arboricity} $\Arb(G)$ is
the maximum of $\FR{|E(H)|}{|V(H)|-1}$ over all subgraphs $H$ with at least 
two vertices.  Generalizing the Nash-Williams Arboricity Theorem, the Nine
Dragon Tree Conjecture asserts that if $\Arb(G)\le k+\frac{d}{k+d+1}$, then $G$
decomposes into $k+1$ forests with one having maximum degree at most $d$.
The conjecture was previously proved for $d=k+1$ and for $k=1$ when $d\le6$.
We prove it for all $d$ when $k\le2$, except for $(k,d)=(2,1)$.
\end{abstract}

\baselineskip 16pt

\section{Introduction}

Throughout this paper, we consider loopless multigraphs; this is the model we
mean when we say ``graph''.  A {\it decomposition} of a graph $G$ consists of
edge-disjoint subgraphs with union $G$.  The {\it arboricity} of $G$, written
$\Upsilon(G)$, is the minimum number of forests needed to decompose it.  The
famous Nash-Williams Arboricity Theorem~\cite{NW} states that
$\Upsilon(G)\le k$ if and only if no subgraph $H$ has more than $k(\C{V(H)}-1)$
edges.

The {\it fractional arboricity} of $G$, written $\Arb(G)$, is
$\max_{H\esub G}\FR{\C{E(H)}}{\C{V(H)}-1}$ (Payan~\cite{Payan}).  
Nash-Williams' Theorem states $\Upsilon(G)= \lceil \Arb(G) \rceil$.  If
$\Arb(G)=k+\eps$ (with $k\in\NN$ and $0 < \eps \le 1$), then $k+1$ forests are
needed.  When
$\eps$ is small, one may hope to restrict the form of the last forest, since
$k$ forests are ``almost'' enough to decompose $G$.
The {\it Nine Dragon Tree (NDT) Conjecture}
asserts that one can bound the maximum degree of the last forest in terms of
$\eps$ and $k$.  Call a graph {\em $d$-bounded} if its maximum degree is at
most $d$.

\begin{conjecture}[\rm NDT Conjecture \cite{MORZ}]\label{ndtconjecture}
If $\Arb(G)\le k+\FR d{k+d+1}$, then $G$ decomposes into $k+1$ forests, one
of which is $d$-bounded.
\end{conjecture}

Our main result implies the NDT Conjecture for $k\le2$, except for the case
$(k,d)=(2,1)$.  That case, and indeed all cases with $d=1$, was proved by
D. Yang~\cite{Yang}.  Hence the proof is now complete for $k\le2$.

We call a decomposition into $k$ unrestricted forests and one $d$-bounded
forest a {\it $(k,d)$-decomposition}.  A graph having a $(k,d)$-decomposition
is {\it $(k,d)$-decomposable}.  Montassier, Ossona de Mendez, Raspaud, and
Zhu~\cite{MORZ} posed the NDT Conjecture, showed that the condition cannot be
relaxed, and proved the conjecture for $(k,d) \in \{(1,1),(1,2)\}$.  Kim,
Kostochka, Wu, West, and Zhu~\cite{KKWWZ} proved the conjecture for $d=k+1$ and
for the case $k=1$ when $d\le6$.  A stronger version of the conjecture requires
the $d$-bounded forest to have at most $d$ edges in each component.  For
$(k,d)=(1,1)$ the conclusions are the same, and for $(k,d)=(1,2)$ the stronger
version was proved in~\cite{KKWWZ}.

A {\it weak $(k,d)$-decomposition} is a decomposition into $k$ forests and one
$d$-bounded graph.  A weaker conjecture states that $\Arb\le k+\FR d{k+d+1}$
guarantees a weak $(k,d)$-decomposition.  Study of the NDT Conjecture was
motivated by problems about weak $(k,d)$-decomposition and
$(k,d)$-decomposition of planar graphs discussed
in~\cite{BKPY,Kleit+,BIS07,BIKS09,BIKS09a,BK+02,G,HH+02,Kleit,Mont,MORZ,WZ}.
These results, some of which are successive reductions of the girth needed to
guarantee $(1,1)$- or $(1,2)$-decompositions of planar graphs, are summarized
in~\cite{KKWWZ}.  Our result implies all of these except the results about
$(2,d)$-decomposition of planar graphs in~\cite{BKPY} and~\cite{G}.

The weaker conjecture was proved for $d>k$ in~\cite{KKWWZ}.  For $d\le k$,
the more restrictive hypothesis $\Arb(G)\le k+\frac{d}{2k+2}$ suffices for
$(k,d)$-decomposability~(\cite{KKWWZ}).  When $d=k+1$ this hypothesis is the
same as that in the NDT, which yields the NDT for $d=k+1$.  Meanwhile, Kir\'aly
and Lau~\cite{KL} showed by other methods that $G$ is weakly
$(k,d)$-decomposable when $\Arb(G)\le k+\FR{d-1}{k+d}$.  For $d<k$, neither of
the results of~\cite{KKWWZ} and~\cite{KL} implies the other.

In the computation of $\Arb(G)$, it suffices to maximize over induced subgraphs.
Letting $G[A]$ denote the subgraph of $G$ induced by a vertex set $A$, and
letting $\E A$ denote $\C{E(G[A])}$, we can rewrite the condition bounding
$\Arb(G)$ as an integer inequality.  In the same format we define a weaker
condition of being {\it $(k,d)$-sparse}:

\begin{center}
\begin{tabular}{c c l}
Condition&&Equivalent constraint (when imposed for $\nul\ne A\esub V(G)$)\\
$\Arb(G)\le k+\FR d{k+d+1}$&&$(k+1)(k+d)\C A-(k+d+1)\E A\ge(k+1)(k+d)$\\
$(k,d)$-sparse&&$(k+1)(k+d)\C A-(k+d+1)\E A\ge k^2$\\
\end{tabular}
\end{center}

In~\cite{KKWWZ}, it was proved that every $(k,d)$-sparse graph is weakly
$(k,d)$-decomposable when $d>k$.  Graphs without weak $(k,d)$-decompositions
were given where the $(k,d)$-sparseness inequality holds for every nonempty
proper vertex set but fails by $1$ on the entire set.  For sufficiency, a
more general model involving ``capacities'' was used to control vertex degrees.

\begin{definition}\label{capacity}
Fix positive integers $k$ and $d$.  A {\it capacity function} on a graph $G$ is
a function $\MAP f{V(G)}{\{0,\ldots,d\}}$.  A {\it $(k,f)$-decomposition} of
$G$ is a decomposition into graphs $F$ and $D$, where $F$ decomposes into $k$
forests and $D$ is a forest having degree at most $f(v)$ at each vertex $v$.
The {\it uniform case} is $f(v)=d$ for all $v\in V(G)$.

In this setting, we define a {\it potential function} $\rho$ on the vertices,
edges, and vertex subsets of $G$.  For a vertex $v$, let
$\rho(v)=(k+1)(k+f(v))$.  For an edge $e$, let $\rho(e)=-(k+1)$ if the
endpoints of $e$ both have capacity $0$, and otherwise $\rho(e)=-(k+1+d)$.
For $A\esub V(G)$, let $\rho(A)=\SM vA \rho(v)+\SM e{E(G[A])} \rho(e)$.
A capacity function $f$ is {\it feasible} if the corresponding potential
function $\rho$ satisfies $\rho(A)\ge k^2$ for every nonempty vertex subset $A$.
\end{definition}


In the uniform case, every vertex has capacity $d$, so every vertex has
potential $(k+1)(k+d)$ and every edge has potential $-(k+1+d)$.  The
feasibility inequality then reduces to precisely the definition of
$(k,d)$-sparseness.  However, in order to obtain $(k,d)$-decomposability, the
hypothesis must be strengthened, because feasibility leaves the possibility of
subgraphs having too many edges to decompose into $k+1$ forests.

\begin{definition}\label{overfull}
A vertex subset $A\esub V(G)$ is {\em overfull} if $\E{A}>(k+1)(\C A-1)$.
\end{definition}

Large overfull sets are forbidden by $(k,d)$-sparseness, but small ones are not.
An overfull set $A$ of size $r$ with $\E{A}=(k+1)(r-1)+1$ satisfies the
$(k,d)$-sparseness inequality when $r\le\frac{k(d+1)}{k+1}$.
For example, $(k,d)$-sparse graphs may have edges with multiplicity $k+2$,
but $(k,d)$-decomposable graphs cannot.  Since $\Arb(G)\le k+\FR d{k+d+1}$
implies that $G$ is $(k,d)$-sparse and has no overfull set, the conjecture
below strengthens the NDT Conjecture. 

\begin{conjecture}\label{feasfullconj}
Fix $k,d\in\NN$.  If $G$ is $(k,d)$-sparse and has no overfull set,
then $G$ is $(k,d)$-decomposable.
\end{conjecture}

Our main result proves a more general statement in the case $k\le2$;
capacity functions facilitate the proof.


\begin{theorem} \label{main}
For $k\le2$ and $d\ge k$, if $f$ is a feasible capacity function defined on
$G$, and $G$ has no overfull set,
then $G$ is $(k,f)$-decomposable.
\end{theorem}

\begin{remark}
The statement for general capacity functions contains the statement for the
uniform case, but the uniform case also implies the general case.  This was
shown in~\cite{KKWWZ} for a similar potential function; we explain it here more
simply.

Capacity $f(v)$ on $v$ can modeled in the uniform case by adding $d-f(v)$
neighbors of $v$, each with $k+1$ edges to $v$ and having no other neighbors.
Each such neighbor forces an edge at $v$ into $D$.  When $u$ with capacity $d$
is added as such a neighbor of $v$ and $f(v)$ is increased by $1$ to allow the
added edge in $D$, the old potential of a set $A$ containing $v$ equals the new
potential of $A\cup\{u\}$ (we gain and lose $(k+1+d)(k+1)$).

Thus feasibility in the general case is equivalent to feasibility of the
corresponding augmented sets in the uniform case, and the existence of the
desired decompositions is also equivalent.  Nevertheless, the general result is
easier to prove: capacity functions facilitate reserving an edge for the
$d$-bounded forest $D$ while reducing the number of edges, by deleting the edge
and reducing the capacity of its endpoints.

Under the potential function used in~\cite{KKWWZ}, every edge has potential
$-(k+1+d)$.  The definition here makes more capacity functions feasible and
hence applies more generally.
\end{remark}

%
%
%
%

Our approach to proving Theorem~\ref{main} is to restrict the form of a
smallest counterexample.  The restrictions we prove are valid for general
$(k,d)$.  For example, in a smallest counterexample the only nonempty proper
vertex subsets with potential at most $k(k+1)$ consist of single vertices with
capacity $0$ (Lemma~\ref{potkk+}).  Furthermore, when $A$ is a proper vertex
subset with $\rho(A)\le k(k+1)+d$ and $|A|\ge2$, every vertex of $A$ having
a neighbor outside $A$ must have capacity $0$ (Lemma~\ref{mainlem}).

We then use discharging, restricting the argument to $d\ge k$.  The initial
charge of each vertex or edge is its potential.  Hence the total charge is at
least $k^2$, but after vertices give their charge to incident edges, all
vertices and edges in an instance satisfying the reductions have nonpositive
charge, if additional constraints on vertices of capacity $0$ or $d$ hold.
Those constraints hold automatically when $k=1$, and additional lemmas in
Section~\ref{sec:k2} show that they also hold when $k=2$ if $d>1$.

\section{General Reduction Lemmas}\label{sec:reduc}

Given fixed $(k,d)$, an \emph{instance} of our problem is a pair $(G,f)$ such
that $G$ has no overfull set and $f$ is feasible on $G$.  We speak of $\rho$ in
Definition~\ref{capacity} as the potential function for the pair $(G,f)$.  We
need to place the instances $(G,f)$ in order (actually a partial order).

\begin{definition}\label{order}
Given two instances $(G,f)$ and $(G',f')$, say that $(G',f')$ is
\emph{smaller than} $(G,f)$ if (1) $\C{E(G')}<\C{E(G)}$ or if (2) $G'=G$ and
$\sum_v f'(v)>\sum_v f(v)$.
\end{definition}

A counterexample is an instance $(G,f)$ such that $G$ has no
$(k,f)$-decomposition.  Throughout our discussion, {\it $(G,f)$ is assumed to
be a smallest counterexample} in the sense of Definition~\ref{order}.  To
restrict the form of $(G,f)$, we construct a smaller instance $(G',f')$ and use
the guaranteed $(k,f')$-decomposition of $G'$ to obtain a $(k,f)$-decomposition
of $G$.  Showing that $(G',f')$ is a smaller instance includes showing that it
has no overfull set and that its potential function $\rho'$ is feasible.

Given an instance $(G,f)$, we write $(G[A],f)$ for the instance where $f$ is
restricted to $A$.

\begin{lemma}\label{induced}
If $A$ is a proper subset of $V(G)$, then $G[A]$ is $(k,f)$-decomposable.
\end{lemma}
\begin{proof}
Vertex deletion does not change edge count or the potential of vertex sets that
remain.  Hence $(G[A],f)$ is a smaller instance and is $(k,f)$-decomposable.
\end{proof}

When $S$ and $T$ are disjoint vertex subsets of a graph, let $[S,T]$ denote the
set of edges with endpoints in $S$ and $T$.  When $S\cup T=V(G)$, the set
$[S,T]$ is an {\it edge cut} of $G$.

\begin{lemma}\label{edgeconn}
$G$ is $(k+1)$-edge-connected (and hence has minimum degree at least $k+1$).
\end{lemma}
\begin{proof}
Let $[S,T]$ be an edge cut of $G$.  By Lemma~\ref{induced}, $G[S]$ and $G[T]$
have $(k,f)$-decomposi\-tions $(F_S,D_S)$ and $(F_T,D_T)$.  If
$\Ccut{[S,T]}\le k$, then these can be combined by adding one edge of
$[S,T]$ to the union of the $i$th forests in $F_S$ and $F_T$, for
$1\le i\le \Ccut{[S,T]}$.
\end{proof}

Let $N_G(v)$ denote the neighborhood of $v$ in $G$.
Let $V_i=\{v\in V(G)\st f(v)=i\}$.

\begin{lemma}\label{degree}
If $f(v)\!<\!d$, then $d_G(v)\ge k\!+\!1\!+\!f(v)$.
If $f(v)\!>\!0$ and $N_G(v)\!\esub\! V_0$, then $f(v)\!=\!d$.
\end{lemma}
\begin{proof}
If either statement fails, then $f(v)<d$.  Raising the capacity of a vertex
with positive capacity does not change the number of edges or the potential of
any vertex subset.  Hence $(G,f')$ is an instance, where $f'(v)=f(v)+1$ and
$f'(u)=f(u)$ for $u\in V(G)-\{v\}$.  By criterion (2), $(G,f')$ is smaller than
$(G,f)$ and hence has a $(k,f')$-decomposition $(F',D')$.

Note that $(F',D')$ is a $(k,f)$-decomposition of $G$ unless $d_{D'}(v)=f(v)+1$.
If also $d_G(v)\le k+f(v)$, then $v$ is isolated in some forest in the
decomposition of $F'$.  Moving one edge of $D'$ incident to $v$ into that
forest yields a $(k,f)$-decomposition of $G$.

If $N_G(v)\esub V_0$, then $v$ is isolated in $D'$ and $(F',D')$ is a
$(k,f)$-decomposition of $G$.
\end{proof}

A vertex subset $A$ is {\it nontrivial} if $2\le \C A\le \C{V(G)}-1$.  In most
applications of the next lemma we take $r=0$ in the statement, in which case
the resulting capacity function $f^*$ is just the restriction of $f$ to $A$.

\begin{lemma}\label{contract}
Let $A$ be a nontrivial vertex set in a graph $H$, with $H'$ obtained by
contracting $A$ to a vertex $z$.  Let $f'(z)=r$ and $f'(v)=f(v)$ for
$v\in V(H')-\{z\}$.  Suppose that $H'$ has a $(k,f')$-decomposition $(F',D')$,
in which $d'(x)$ for $x\in A$ is the number of edges incident to $x$ that
become edges of $D'$ incident to $z$.  If also $H[A]$ has a
$(k,f^*)$-decomposition $(F^*,D^*)$, where $f^*(x)=f(x)-d'(x)$ for $x\in A$,
then $H$ is $(k,f)$-decomposable.
\end{lemma}
\begin{proof}
Viewing edges at $z$ in $G'$ as the corresponding edges in $G$, define $(F,D)$
by letting $F$ consist of $k$ subgraphs, where the $i$th subgraph
is the union of the $i$th subgraphs in the decompositions of $F'$ and
$F^*$ into forests.  Similarly let $D=D'\cup D^*$.

The $k$ subgraphs in the decomposition of $F$ are forests, as is $D$, because
any cycle would contract to a cycle in the corresponding forest in $F'$ or
$D'$.  Also, $d_D(v)\le f(v)$ for all $v$, because for $v\in A$ the number of
edges incident to $v$ in $D^*$ is at most $f(v)-d'(v)$.
\end{proof}


When counting the edges induced by a vertex set $A$ in a graph $H$
other than $G$, we use $\E A_H$ to avoid confusion.

\begin{lemma}\label{nocap0}
Let $f$ be a feasible capacity function on some graph $H$, and consider
$A\esub V(H)$ with $\C A\ge2$.  Let $A_0=\{v\in A\st f(v)=0\}$.  If
$\E A_H\ge(k+1)(\C A-1)$, then $A_0=\nul$.
\end{lemma}
\begin{proof}
If $\C{A_0}\ge1$, then $\rho(A_0)\ge k^2$ is equivalent to
$\E{A_0}_H\le k(\C{A_0}-1)$.  We compute
\begin{align*}
\rho(A)&\le (k+1)(k+d)\C A-(k+1+d)\E A_H - (k+1)d\C{A_0}+d\E{A_0}_H\\
&\le -(k+1)\C A+(k+1+d)(k+1)-d\C{A_0}-dk \le k^2-1,
\end{align*}
where the last inequality uses $\C A\ge2$.  This contradicts the feasibility of
$f$.
\end{proof}


\begin{definition}\label{Acon}
For a nontrivial vertex set $A\esub V(G)$, the \emph{$A$-contraction} of
an instance $(G,f)$ is the pair $(G',f')$, where $G'$ is be obtained from $G$
by shrinking $A$ to a single vertex $z$, and $f'$ is defined on $G'$ by
$f'(z)=0$ and $f'(v)=f(v)$ for $v\in V(G')-\{z\}$.  Edges of $G$ induced by 
$A$ disappear in $G'$, and an edge joining $x\in A$ and $y\notin A$ in $G$ 
becomes an edge joining $z$ and $y$ in $G'$.
\end{definition}

\begin{lemma}\label{contract0}
Let $(G',f')$ be the $A$-contraction of an instance $(G,f)$, where $A$ is a
nontrivial subset of $V(G)$.  If $f'$ is feasible, then $(G,f)$ cannot be a
smallest counterexample.
\end{lemma}
\begin{proof}
If $f'$ is feasible, then $G'$ has no overfull set containing a vertex of
capacity $0$, by Lemma~\ref{nocap0}.  Hence $G'$ has no overfull set containing
$z$.  Also $G'$ has no overfull set not containing $z$, since $G$ has no
overfull set.  Hence $(G',f')$ is a smaller instance than $(G,f)$.
If $(G,f)$ is a smallest counterexample, then $G'$ has a $(k,f')$-decomposition
$(F',D')$.  By Lemma~\ref{contract}, $G$ is $(k,f)$-decomposable and is not
a counterexample.
\end{proof}

Proving feasibility for the potential function $\rho'$ of the smaller instance
$(G',f')$ means proving $\rho'(A')\ge k^2$ for $A'\esub V(G')$.  We do not need
to mention subsets $A'$ such that $G'[A']=G[A']$ and $f'(v)=f(v)$ for all
$v\in A'$.

\begin{lemma}\label{potkk+}
If $\nul\ne A\subset V(G)$, then $\rho(A)>k(k+1)$ unless $A$ consists of a
single vertex with capacity $0$.
\end{lemma}
\begin{proof}
Suppose $\rho(A)\le k(k+1)$.  If $A$ is not a single vertex with capacity $0$,
then $\C A>1$.  Let $(G',f')$ be the $A$-contraction of $(G,f)$.

To prove that $f'$ is feasible, consider $A'\esub V(G')$ with $z\in A'$.  Let
$A^*=(A'-\{z\})\cup A$.  Let $E'={[A-V_0,A'-\{z\}\cap V_0]}$.  In moving from
$G$ to $G'$, the potential of each edge in $E'$ changes from $-(k+1+d)$ to
$-(k+1)$; other edges keep the same potential.  Thus
$$\rho'(A') = \rho(A^*)-\rho(A)+\rho(\{z\})+2\C{E'}\ge \rho(A^*)\ge k^2$$
under the assumption $\rho(A)\le k(k+1)$, since $\rho(\{z\})=(k+1)k$.

Since $f'$ is feasible, Lemma~\ref{contract0} applies, and $G$ is
$(k,f)$-decomposable.
\end{proof}

We say that a set $A$ is \emph{full} when $\E A\ge(k+1)(\C A-1)$.  A full set
of size $2$ is an edge with multiplicity at least $k+1$.  The exclusion of
vertices with capacity $0$ from full sets (Lemma~\ref{nocap0}) helps us to
exclude all full sets when $(G,f)$ is a smallest counterexample. 

\begin{lemma}\label{nofull}
$G$ has no full set $A$ with $\C A\ge2$.
\end{lemma}
\begin{proof}
Since $G$ has no overfull set, we may assume $\E A=(k+1)(\C A-1)$.  By
Lemma~\ref{nocap0}, $A\cap V_0=\nul$.  Hence all edges in $G[A]$ have potential
$-(k+1+d)$.  We compute
$$\rho(A)\le (k+1)(k+d)\C A-(k+1+d)(k+1)(\C A-1) = (k+1)(k+1+d-\C A).$$

If $A=V(G)$, then let $i=\min\{f(v)\st v\in V(G)\}$.  Accounting for one vertex
with capacity $i$ reduces the bound by $(k+1)(d-i)$ to
$\rho(V(G))\le (k+1)(k+1+i-\C{V(G)})$.  Now
$k^2\le \rho(V(G))\le k^2-1+(k+1)(2+i-\C{V(G)})$ yields
$i\ge\C{V(G)}-1$.  This means $d\le\C{V(G)}-1$, so there is no restriction on
the degrees of vertices in the $(k+1)$th forest, and the prohibition of
overfull sets ensures that $G$ decomposes into $k+1$ forests.

Hence we may assume $A\ne V(G)$.  Continuing the computation and using
$\C A\ge2$,
$$\rho(A)\le (k+1)(k+1+d-\C A) \le (k+1)(k-1+d)=k^2+d(k+1)-1.$$
Now let $\ell=\rho(A)-k^2$ and $m=\FL{\FR \ell{k+1}}$; the bound on $\rho(A)$ yields
$m<d$.  If $f(x)\le m$ for some $x\in A$, then adjusting the potential for
this vertex yields
\begin{align*}
\rho(A)&\le(k+1)(k+1+d-\C A)-(k+1)(d-m)\le (k+1)(k-1+m)\le k^2-1+\ell,
\end{align*}
again using $\C A\ge2$.  The contradiction implies $f(x)>m$ for all $x\in A$.

Form $G'$ by contracting $A$ into a new vertex $z$.  Let $f'(z)=m$ and
$f'(v)=f(v)$ for $v\in V(G)-A$.  For any set $A'$ with $z\in A'\esub V(G')$,
replacing $z$ with $A$ adds $\C A-1$ vertices and $(k+1)(\C A-1)$ edges.  Hence
$A'$ is overfull in $G'$ if and only if $A\cup A'$ is overfull in $G$.  Since
$G$ has no overfull set, $G'$ has no overfull set.

For $z\in A'\esub V(G')$, let $A^*=(A'-\{z\})\cup A$.  Since $f(z)>0$ and
$A\cap V_0=\nul$, every edge incident to $z$ in $G'$ has the same potential as
the corresponding edge in $G$.  We compute
\begin{align*}
\rho'(A')-\rho(A^*)&\ge \rho'(z)-\rho(A)=(k+1)(k+m)-k^2-\ell\\
&\ge k+(k+1)\left(\FR \ell{k+1}-1\right)-\ell=-1
\end{align*}
However, Lemma~\ref{potkk+} yields $\rho(A^*)> k(k+1)$, so
$\rho'(A')\ge k(k+1)$.  Hence $(G',f')$ is feasible.

Since $(G',f')$ is smaller than $(G,F)$, there is a $(k,f')$-decomposition
$(F',D')$ of $G'$, with $d_{D'}(z)\le m$.  Let $f^*(x)=f(x)-g(x)$ for $x\in A$,
where $g(x)$ is the number of edges joining $x$ to $V(G)-A$ that become edges
of $D'$ when $A$ is contracted (an edge in $D'$ may have several choices for
which vertex it is assigned to).

Since $f'(z)=m$ and $f(x)>m$ for $x\in A$, we have $f^*(x)>0$ for $x\in A$,
so $f^*$ is a capacity function on $G[A]$.  Since $G$ has no overfull sets,
$\E A-\E X\ge(k+1)(\C A-\C X)$ for $X\esub A$.  The potential of $X$ is
smallest in comparison to that of $A$ when all vertices of $A-X$ have capacity
$d$ and all edges of $D'$ incident to $z$ arise from edges incident to $X$.
Hence
\begin{align*}
\rho^*(X)&\ge\rho(A)-m(k+1)-(k+1)(k+d)(\C A-\C X)+(k+1+d)(k+1)(\C A-\C X)\\
&=\rho(A)-m(k+1)+(k+1)(\C A-\C X)\ge \rho(A)-\ell=k^2,
\end{align*}
using $\C A\ge\C X$ and the definition of $m$ in the last step.

Hence $(G[A],f^*)$ is a smaller instance, and $G[A]$ is $(k,f^*)$-decomposable.
By Lemma~\ref{contract}, $G$ is $(k,f)$-decomposable.
\end{proof}

Lemma~\ref{nofull} forbids edges with multiplicity $k+1$.  Within the set $V_0$,
we can reduce the multiplicity further.  In particular, when $k=1$ a minimal
counterexample must be a simple graph in which $V_0$ is an independent set.
One can also prove that no vertex of $V_0$ has an incident edge of multiplicity
$k$, but we will not need that.

\begin{lemma}\label{cap0}
No two vertices of $V_0$ are joined by $k$ edges.
\end{lemma}
\begin{proof}
If $x$ and $y$ are vertices of capacity $0$ joined by $k$ edges,
then $\rho(\{x,y\})=2k(k+1)-k(k+1)=k(k+1)$, contradicting Lemma~\ref{potkk+}.
%
\end{proof}

\begin{lemma}\label{onedown}
For $x\in V(G)$ with $f(x)>0$, every proper induced subgraph of $G$ containing
$x$ has a $(k,f)$-decomposition $(F,D)$ such that $d_D(x)<f(x)$.
\end{lemma}
\begin{proof}
Consider $A$ with $x\in A\subset V(G)$.  If $f(x)>0$, then define $f'$ by
$f'(x)=f(x)-1$ and $f'(v)=f(v)$ for $v\in A-\{x\}$.  For $x\in A'\esub A$ with
$\C{A'}\ge2$, we have $\rho'(A')\ge\rho(A')-(k+1)\ge k^2$, by
Lemma~\ref{potkk+} (if $f(x)=1$, then the inequality may be strict).
Also $A$ has no overfull subset.  Since $(G[A],f')$ is smaller than $(G,f)$,
it has a $(k,f')$-decomposition, which is a $(k,f)$-decomposition such that
$d_D(x)<f(x)$.
\end{proof}

\begin{lemma}\label{adjpos}
If $f(u)>0$ and $d_G(u)=k+1$, then $N_G(u)\esub V_0$.
\end{lemma}
\begin{proof}
If there exists $x\in N_G(u)$ with $f(x)>0$, then by Lemma~\ref{onedown}
$G-u$ has a $(k,f)$-decomposition $(F,D)$ such that $d_{D'}(x)<f(x)$.
Add one edge with endpoints $\{x,y\}$ to $D$ and the other edges at $x$ to
distinct forests in $F$ to complete a $(k,f)$-decomposition of $G$.
\end{proof}

Recall that $d_G(x)\ge k+1+f(x)$ when $0<f(x)<d$ (Lemma~\ref{degree}).
When $f(x)=d$, our lower bound on degree is weaker.

\begin{lemma}\label{2inV0}
If $x\in V(G)$ and $\C{N(x)\cap V_0}\ge2$, then $d_G(x)>k+1$.
\end{lemma}
\begin{proof}
By Lemma~\ref{edgeconn}, $d_G(x)\ge k+1$; consider equality.  Since $G$ has no
nontrivial full set, adding an edge joining vertices $y,z\in N_G(x)\cap V_0$ to
form $G'$ from $G-x$ creates no overfull
set.  The added edge decreases the potential of any set $A$ containing
$\{y,z\}$ by $k+1$.  By Lemma~\ref{potkk+}, $\rho(A)>k(k+1)$ if
$\{y,z\}\esub A\esub V(G)-\{x\}$, and therefore $\rho'(A)\ge k^2$.

Hence $(G',f)$ is an instance smaller than $(G,f)$, and $G'$ has a
$(k,f)$-decomposition $(F',D')$.  Since $y,z\in V_0$, the added edge $yz$ lies
in $F'$.  Replace it in its forest with the $y,z$-path of length $2$ through
$x$.  The remaining $k-1$ edges at $x$ can be added to the $k-1$ other forests
in $F'$.
\end{proof}

\begin{lemma}\label{fullcap}\label{cap00}
If $x\in V(G)$, then $d_G(x)\ge k+2$ unless $f(x)=0$ and
$\C{N_G(x)\cap V_0}\le 1$.
\end{lemma}
\begin{proof}
If $\C{N_G(x)\cap V_0}\ge2$, then Lemma~\ref{2inV0} applies.  When $f(x)=0$
there is nothing further to show.  If $f(x)>0$, then Lemma~\ref{edgeconn}
yields $d_G(x)\ge k+1$, Lemma~\ref{adjpos} yields $N_G(u)\esub V_0$ when
$d_G(x)=k+1$, and Lemma~\ref{nofull} prevents all $k+1$ incident edges from
going to a single neighbor; hence $\C{N_G(x)\cap V_0}\ge2$.
\end{proof}

\section{Final reductions}\label{sec:mainlem}

Our final reductions restrict the edges leaving $A$ when $A$ has small
potential.  For $A\esub V(G)$, let the {\it boundary} $\partial A$ denote the
set of vertices in $A$ having a neighbor outside $A$.

\begin{lemma}\label{cobound}
If $A$ is a nontrivial subset of $V(G)$ such that $\rho(A)\le k(k+1)+d$,
then no edge joining $\partial A$ to $V(G)-A$ has positive capacity at both
endpoints.
\end{lemma}
\begin{proof}
Suppose otherwise, and choose $A$ among the counterexamples with smallest
potential.  Let $xy$ be an edge with $x\in A$, $y\notin A$, and $f(x),f(y)>0$.
Form $G'$ from $G$ by deleting this edge and then contracting $A$ to a single
vertex $z$.  Define $f'$ on $G'$ by $f'(z)=0$ and $f'(y)=f(y)-1$ and
$f'(v)=f(v)$ for $v\in V(G)-A-\{y\}$.

We first prove that $f'$ is feasible.  Consider $A'\esub V(G')$.  If
$y,z\notin A'$, then $\rho'(A')=\rho(A')$.  If $A'$ contains $y$ and not $z$,
then $\rho'(A')=\rho(A')-(k+1)\ge k^2$, by Lemma~\ref{potkk+}.  When $z\in A'$,
let $A^*=(A'-z)\cup A$.  If $y,z\in A'$, then since $f'(y)=f(y)-1$ and one copy
of $xy$ in $G$ is missing from $G'$, we obtain $\rho'(A')\ge\rho(A^*)\ge k^2$
from
\begin{align*}
\rho'(A')-\rho(A^*)&\ge -\rho(A)+\rho'(z)-(k+1)-\rho(xy)\\
&\ge -[k(k+1)+d]+k(k+1)-(k+1)+(k+1+d)=0.
\end{align*}
If $A'$ contains $z$ and not $y$, then $\rho(A^*)\le\rho(A)+\rho'(A')-\rho(z)$.
If $\rho'(A')<k^2$, then $A'\ne\{z\}$ and $\rho(A^*)< \rho(A)-k$.  Since the
edge $xy$ joining $\partial A^*$ to $V(G)-A$ has positive capacity at both
endpoints, and $A^*$ is a nontrivial set, this contradicts our choice of $A$ as
a counterexample with smallest potential.  We conclude $\rho'(A')\ge k^2$.

Hence $f'$ is feasible.  Since $G$ has no overfull set, an overfull set must
contain $z$.  However, since $f'$ is feasible on $G'$ and $f'(z)=0$,
Lemma~\ref{nocap0} implies that no overfull set in $G'$ contains $z$.
Hence $G'$ contains no overfull set.

Since $(G',f')$ is smaller than $(G,f)$, we now have a $(k,f')$-decomposition
of $G'$.  By Lemma~\ref{onedown}, $G[A]$ has a $(k,f)$-decomposition $(F^*,D^*)$
such that $d_{D^*}(x)<f(x)$.  By Lemma~\ref{contract}, the two decompositions
combine to form a $(k,f)$-decomposition $(F,D)$ of $G-xy$ that becomes a
$(k,f)$-decomposition of $G$ by adding a copy of the edge $xy$ to $D$.
\end{proof}

\begin{lemma}\label{mainlem}
If $A$ is a nontrivial subset of $V(G)$ such that $\rho(A)\le k(k+1)+d$, then
$\partial A\esub V_0$.
\end{lemma}
\begin{proof}
If not, then let $A$ be a largest nontrivial subset with $\rho(A)\le k(k+1)+d$
and $\partial A\nosub V_0$.  Choose $x\in\partial A$ with $f(x)>0$, and
choose $y\in N_G(x)-A$.  By Lemma~\ref{cobound}, $f(y)=0$.

Let $(G',f')$ be the $A$-contraction of $(G,f)$.  If $f'$ is feasible, then $G$
is $(k,f)$-decomposable, by Lemma~\ref{contract0}.
Consider $A'\esub V(G')$ with $z\in A'$, and let $A^*=(A'-\{z\})\cup A$.

If $\rho'(A')<k^2$, then $\rho(A^*)\le \rho'(A')-\rho'(z)+\rho(A)<\rho(A)$.
This contradicts the choice of $A$ if $A^*$ is nontrivial and
$\partial A^*\nosub V_0$.  Hence we may assume $A^*=V(G)$ or
$\partial A^*\esub V_0$.  In either case, since $x\in A$ and $f(x)>0$, we
must have $y\in A'$.  Now $\rho(xy)=-(k+1+d)$, but $\rho'(zy)=-(k+1)$, so
$\rho'(zy)-\rho(xy)=d$.  Hence
$$\rho'(A')\ge \rho(A^*)-\rho(A)+\rho'(z)+d\ge\rho(A^*),$$
by the hypothesis $\rho(A)\le k(k+1)+d$.  We conclude that $f'$ is feasible, as
desired.
\end{proof}

In order to clarify which part of our proof requires $k\le 2$ and suggest
further directions, we now give a discharging argument to show what remains to
be excluded when $d$ is large.

\begin{theorem}\label{disch}
Let $(G,f)$ be a smallest counterexample, and put $h(v)=\Ccut{[\{v\},V_0]}$.
If $d\ge k$ or $k\le2$, then some $v\in V(G)$ satisfies\\
$\phantom{0}$\quad(1) $f(v)=0$ with $h(v)>2(d_G(v)-k-1)\FR k{k-1}$, or\\
$\phantom{0}$\quad(2) $f(v)=d$ with
$h(v)<\FR{(2k+2-d_G(v))(k+1+d)-2(k+1)}{d+1-k}$.\\
In particular, (1) requires $d_G(v)<2k$, and (2) requires $d_G(v)<2k+2$.
\end{theorem}
\begin{proof}
A smallest counterexample has all the properties derived in the prior lemmas,
which impose no restriction on $(k,d)$.  We use discharging to show that for
such an instance, the total potential is nonpositive when vertices as specified
above are also forbidden.

Give each vertex and edge initial charge equal to its potential.  Hence the
total charge is at least $k^2$.  The edges now take charge from their
endpoints by the following rules:

\medskip
Rule 1: Every edge $xy$ with $f(x)=0$ and $f(y)>0$ takes $k$ from $x$
and $d+1$ from $y$.

Rule 2: Every edge joining vertices not in $V_0$ takes $(k+1+d)/2$ from each
endpoint.

Rule 3: Every edge joining vertices in $V_0$ takes $(k+1)/2$ from each endpoint.

\medskip
By construction, edges end with charge $0$.  It suffices to show that all
vertices also end with nonpositive charge.  Consider $v\in V(G)$.
By Lemma~\ref{edgeconn}, $d_G(v)\ge k+1$.

If $f(v)=0$, then $v$ loses charge $h(v)\FR{k+1}2+(d_G(v)-h(v))k$.  This
is at least $k(k+1)$ if and only if $h(v)\le 2(d_G(v)-k-1)\FR k{k-1}$.  In
the special case $d_G(v)\ge 2k$, violating this inequality requires
$h(v)>2(\FR12 d_G(v)-1)\FR{d_G(v)/2}{d_G(v)/2-1}=d_G(v)$, but always
$h(v)\le d_G(v)$.

Each vertex with positive charge loses at least $\FR{k+1+d}2$ to each incident
edge, by Rule 1 or Rule 2 (since $d+1\ge k$).  If $1\le f(v)< d$, then
$d_G(v)\ge k+1+f(v)$, by Lemma~\ref{degree}.  Hence $v$ loses at least 
$\FR{k+1+d}2(k+1+f(v))$.  If $d\ge k+1$, then $\FR{k+1+d}2\ge k+1$, and the
lower bound $(k+1)(k+1+f(v))$ on the lost charge exceeds the initial charge
$(k+1)(k+f(v))$.  The remaining case is $k=d>f(v)$, where
\begin{center}
$(k+1)(k+f(v))-\FR{k+1+d}2(k+1+f(v))=\FR{1}2(k+f(v))-k-\FR12 <0$.
\end{center}

Finally, consider $f(v)=d$; Lemma~\ref{fullcap} gives $d_G(v)>k+1$.
The charge lost by $v$ is $h(v)(d+1)+(d_G(v)-h(v))\FR{k+1+d}2$.
For $v$ to reach nonpositive charge, this must be at least $(k+1)(k+d)$.
The inequality simplifies to $h(v)\ge\FR{(2k+2-d_G(v))(k+1+d)-2(k+1)}{d+1-k}$.
This always holds when $d_G(v)\ge 2k+2$, and when $d_G(v)=2k+1$ we only need
$h(v)\ge 1$ or $d\le k+1$.
\end{proof}

\section{The Case $k\le2$}\label{sec:k2}

Note that Theorem~\ref{disch} applies for all $(k,d)$ with $k\le2$.  Our task
is to prohibit the exceptions in Theorem~\ref{disch} when $(G,f)$ is a smallest
counterexample.

First suppose $f(v)=0$.  If $d_G(v)\ge2k$, then there is no problem, which
suffices when $k=1$.  For $k=2$ we have $\delta(G)\ge3$, so we may assume
$d_G(v)=3$.  To avoid the exception, we must show that all neighbors of $v$
have positive capacity.

Now suppose $f(v)=d$.  By Lemma~\ref{fullcap}, $d_G(v)>k+1$.
By Theorem~\ref{disch}, we may assume $d_G(v)\le 2k+1$.  For $k=1$, we thus
have $d_G(v)=3$ and only need one neighbor in $V_0$.  When $k=2$, we may assume
$d_G(v)\in\{4,5\}$.  When $d_G(v)=5$, we only need one neighbor in $V_0$.
When $d_G(v)=4$, we need $h(v)\ge3$ if $d\ge3$, but $h(v)=4$ if $d=2$.

\begin{lemma}\label{k+2}
For $k\le 2$, let $(G,f)$ be a smallest counterexample.
If $d_G(x)=k+2$ with $f(x)\ge2$, then $N_G(x)\esub V_0$.
\end{lemma}
\begin{proof}
If the conclusion fails, then $x$ has a neighbor $y$ with $f(y)>0$.  Since $G$
has no overfull set, $x$ has at least two neighbors; choose $u\in N_G(x)-\{y\}$.
Let $u'$ be a third vertex of $N_G(x)$, if possible; otherwise, let $u'=y$.
Form $G'$ from $G-x$ by adding one copy of the edge $uu'$.  Define $f'$ by
$f'(y)=f(y)-1$ and $f'(v)=f(v)$ for $v\in V(G)-\{x,y\}$.  By
Lemma~\ref{nofull}, $G$ has no full set, and hence $G'$ has no overfull set.

To show that $f'$ is feasible, we need $\rho'(A')\ge k^2$ for $A'\esub V(G')$.
By Lemma~\ref{potkk+}, $\rho(A')>k(k+1)$ if $\C{A'}\ge2$.  Even with
$\rho'(y)=\rho(y)-(k+1)$, we thus have $\rho'(A')\ge k^2$ unless $u,u'\in A'$.
If $y\notin A'$, then $\rho'(A')=\rho(A')+\rho'(uu')$.
Again we have $\rho'(A')\ge k^2$ unless $\rho'(uu')=-(k+1+d)$ and 
$\rho(A')\le k(k+1)+d$.  Since $u,u'\in N_G(x)$, we have $u,u'\in\partial A'$.
Lemma~\ref{mainlem} then requires $u,u' \in V_0$, so $\rho'(uu')=-(k+1)$
and again $\rho'(A)\ge k^2$.

Hence $\rho'(A')<k^2$ requires $y,u,u'\in A'$.  Let $r$ be the number of edges
joining $x$ to $A'$; each has potential $-(k+1+d)$, since $f(x)>0$.
Since $\rho(y)-\rho'(y)=k+1$, we have
\begin{align*}
\rho'(A')
&\ge \rho(A'\cup\{x\})-\rho(x)-(k+1)+\rho'(uu')+r(k+1+d)\\
&\ge\rho(A'\cup\{x\})-(k+1)(k+d+1)+(r-1)(k+1+d).
\end{align*}

It thus suffices to show $r\ge k+1$ when $\rho(A'\cup\{x\})>k(k+1)+d$ and
$r=k+2$ otherwise.
If $\rho(A'\cup\{x\})\le k(k+1)+d$ and $A'\ne V(G)$, then 
$\partial(A'\cup\{x\})\esub V_0$, by Lemma~\ref{mainlem}.
Since $f(x)>0$, this requires $N_G(x)\esub A'$, which also holds if
$A'\cup\{x\}=V(G)$.  Hence $r=k+2$.

When $\rho(A'\cup\{x\})>k(k+1)+d$, we only need $r\ge k+1$.  Here we use
$k\le2$.  If $y\ne u'$, then $x$ has three neighbors in $A'$.  If $y=u'$,
then $N_G(x)=\{u,y\}\esub A'$, and $r=k+2$.

Hence $(G',f')$ is a smaller instance, and $G'$ has a $(k,f')$-decomposition
$(F',D')$.  Since $f(x)\ge2$, the added edge $uu'$ can be replaced in its
forest by a $u,u'$-path $P$ of length $2$ through $x$.  Adding the remaining
$k$ edges at $x$ to the other forests will yield a $(k,f)$-decomposition of
$G$.  When $P$ is added to a forest other than $D'$, we must add one of the
remaining edges to $D'$.  This edge can be $xy$ if $y\ne u'$, since
$f'(y)<f(y)$ and $f(x)>0$.  If $y=u'$, then $N_G(x)=\{y,u\}$; since there is no
full set of size $2$ (Lemma~\ref{nofull}), $G$ has two copies of the edge $xy$.
Hence also in this case a copy of $xy$ is not absorbed by $P$ and can be added
to $D'$.
\end{proof}

\begin{corollary}\label{cork1}
The NDT Conjecture is true when $k=1$.
\end{corollary}
\begin{proof}
When $k=1$, Lemma~\ref{cap0} implies that all neighbors of vertices with
capacity $0$ have positive capacity, so type (1) exceptions in
Theorem~\ref{disch} do not occur.  Type (2) exceptions require
$f(x)=d$ and $d_G(x)=3$, as remarked earlier.  The inequality for the
exception then reduces to $h(x)<1-2/d$, but Lemma~\ref{k+2} yields
$h(x)=d_G(x)=3$.
\end{proof}

In Corollary~\ref{cork1}, one can alternatively examine the discharging
directly.  By Lemma~\ref{k+2}, a vertex $x$ with $f(x)=d$ and $d_G(x)=k+2$
loses charge $(k+2)(d+1)$ when $d\ge2$ and $k\le2$, which exceeds its initial
charge $(k+1)(k+d)$.  When $(k,d)=(1,1)$, Lemma~\ref{k+2} does not apply, but
$x$ then loses at least $\FR32$ to each incident edge, and its initial charge
is $4$.

When $k=2$, again Theorem~\ref{disch} applies.  As we have noted, we must
prohibit $3$-vertices in $V_0$ with neighbors in $V_0$.  For vertices of
capacity $d$, we have noted that only $d_G(x)\in\{4,5\}$ is of concern,
Lemma~\ref{k+2} takes care of $d_G(x)=4$, and when $d_G(x)=5$ we only
need one neighbor in $V_0$.  We consider these remaining cases in two lemmas.

\eject

\begin{lemma}\label{k2cap0}
For $k=2$, a $3$-vertex in $V_0$ has no neighbor in $V_0$.
\end{lemma}
\begin{proof}
Suppose $d_G(x)=3$ with $f(x)=0$ and $N_G(x)\cap V_0\ne\nul$.  Since $G$ has no
full set (Lemma~\ref{nofull}), $\C{N_G(x)}\ge2$.  Let $u$ be a neighbor of $x$
in $V_0$, and let $u'$ be another neighbor of $x$.  Form $G'$ by adding to
$G-x$ an edge joining $u$ and $u'$.  Since $G$ has no full subgraph, $G'$
has no overfull subgraph.

Let $f'$ be the restriction of $f$ to $V(G')$.  Always $\rho'(A')=\rho(A')$
unless $u,u'\in A'$.  If $\rho'(A')<k^2$, then
\begin{align*}
\rho(A'\cup\{x\})&\le \rho'(A')+\rho(x)+\rho(xu)+\rho(xu')-\rho(uu')\\
&\le \rho'(A')+k(k+1)-(k+1)\le 2(k-1)(k+1)
\end{align*}
When $k=2$, we have $2(k-1)(k+1)=k(k+1)$.  This contradicts Lemma~\ref{potkk+}
unless $A'\cup\{x\}=V(G)$.  In that case we add the potential of the third edge
at $x$, obtaining $\rho(A'\cup\{x\})\le (2k-3)(k+1)$, again a contradiction
when $k=2$.  Thus in all cases $\rho'(A')\ge k^2$.

Hence $(G',f')$ is a smaller instance, and $G'$ has a $(2,f')$-decomposition
$(F',D')$.  Since $f(u)=0$, the added edge $uu'$ lies in a forest in $F'$.
Replace it with a path of length $2$ through $x$, and add the third edge at $x$
to the other forest to complete a $(2,f)$-decomposition of $G$.
\end{proof}

Hence it remains only to consider $5$-vertices with capacity $d$.  Although it
is possible to prove that all neighbors of such a vertex $v$ lie in $V_0$, by
Theorem~\ref{disch} we only need the weaker conclusion that some neighbor is in
$V_0$.  That is, $5=2k+2-i$ with $i=1$, and by condition (2) in
Theorem~\ref{disch} it suffices to show $\Ccut{[\{v\},V_0]}\ge1$.
The proof in the next lemma is valid only for $d\ge3$.

\begin{lemma}\label{k25a}
For $k=2$ and $d\ge3$, if $f(v)=d$ and $d_G(v)=5$, then
$N_G(v)\cap V_0\ne\nul$.
\end{lemma}
\begin{proof}
Let $x$ be a $5$-vertex with capacity $d$, and let $U=N_G(x)$ and
$U'=U\cup\{x\}$.  Suppose $N_G(x)\cap V_0=\nul$.  Since $G$ has no edge with
multiplicity at least $3$, we have $\C U\ge 3$.  If equality holds, then some
$u\in U$ is the endpoint of at least two edges incident to $x$.

Form $G'$ from $G-x$ by adding a matching on $U$ if $\C U\ge 4$, and
adding an edge from $u$ to each other vertex of $U$ if $\C U=3$.  For each
endpoint of each added edge, reserve an edge joining it to $x$, thereby
reserving four of the five edges incident to $x$ (if $\C U=3$, then two copies
of $ux$ are reserved).  Define $f'$ on $V(G')$ by $f'(y)=f(y)-1$ and
$f'(v)=f(v)$ for $v\in V(G')-\{y\}$, where $y$ is the endpoint in $U$ of the
unreserved edge at $x$.

Since $G$ has no full set, an overfull set $A'$ in $G'$ must contain the
endpoints of both added edges.  The set $A'\cup\{x\}$, which has one more
vertex and induces at least $\E{A'}_{G'}+4$ edges, is then full in $G$, a
contradiction.  (This argument uses $k\le5$.)  Hence $G'$ has no overfull set.  

Now suppose $\rho'(A')<k^2=4$.  By Lemma~\ref{potkk+}, $G'[A']$ must contain
at least one added edge.  If it contains one added edge $e$ but not the vertex
$y$, then $\rho(A')\le k(k+1)+d$; the edges joining $e$ to $x$ then contradict
Lemma~\ref{cobound}.

If $G[A']$ contains exactly one added edge and also $y$, then $A'\cup \{x\}$
induces at least three edges incident to $x$.  Now
\begin{align*}
\rho(A'\cup\{x\})&\le \rho'(A')+\rho(x)-3(k+1+d)+(k+1+d)+(k+1)\\
&=\rho'(A')+(k+1)(k-1)+d(k-1)\le 6+d=k(k+1)+d,
\end{align*}
using $k=2$.  Again we contradict Lemma~\ref{cobound}.  If $G'[A']$ contains
both added edges, then moving to $G[A'\cup\{x\}]$ loses two edges instead of
one in $A'$ but also gains four edges instead of three at $x$.  If $A'$ also
contains $y$, then the bound increases by $k+1$ for $f(y)$ but decreases by
$k+1+d$ for inducing the fifth edge at $x$.  Hence in each case we obtain
$\rho'(A')\ge k^2$ by essentially the same contradiction.

Hence $(G',f')$ is a smaller instance, and $G'$ has a $(k,f')$-decomposition
$(F',D')$.  If the two added edges lie in distinct forests in the decomposition,
then replace them by paths of length $2$ through $x$ with the same endpoints,
and add the edge $xy$ to the third forest.  This causes no problem when the 
third forest is $D'$, since $f'(y)=f(y)-1$.

If the two added edges lie in the same forest, then deleting them yields at
least three components in that forest, with three endpoints of the added edges
in distinct components.  Extend that forest by edges from $x$ to those three
(distinct) specified vertices.  Add the remaining two edges at $x$ to the other
two forests.  Again, if the forest containing the two specified edges is not
$D'$, then the unreserved edge $xy$ can be added to $D'$.
\end{proof}

The last step in this proof is not valid for $d=2$, because we may be giving $x$
three incident edges in the $d$-bounded forest.  Fortunately, when $d=2$ we
do not need the conclusion of Lemma~\ref{k25a}; the discharging is always
strong enough.

\begin{corollary}
The NDT Conjecture is true when $k=2$ and $d\ge2$.
\end{corollary}
\begin{proof}
As we have remarked, Lemma~\ref{k25a} completes the proof for $k=2$ and $d\ge3$.
When $d=k=2$, a vertex with capacity $2$ has potential $12$, and it loses
charge at least $5/2$ along every edge by the rules in Theorem~\ref{disch}.
Hence a $5$-vertex loses at least $12.5$ and ends with negative charge.
The rest of the proof remains the same as for $d\ge3$.
\end{proof}

\centerline{{\bf\large Acknowledgment}}
We thank Hehui Wu for several useful comments.


{\small

\begin{thebibliography}{99}
\frenchspacing
\parskip=.00ex

\bibitem{BKPY}
J. Balogh, M. Kochol, A. Pluh\'ar, and X. Yu,
\newblock Covering planar graphs with forests.
\newblock {\em J. Combin. Theory Ser. B} 94 (2005), 147--158.

\bibitem{Kleit+}
A. Bassa, J. Burns, J. Campbell, A. Deshpande, J. Farley, M. Halsey, S.-Y. Ho,
D. Kleitman, S. Michalakis, P.-O. Persson, P. Pylyavskyy, L. Rademacher,
A. Riehl, M. Rios, J. Samuel, B.E. Tenner, A. Vijayasaraty, and L. Zhao,
\newblock Partitioning a planar graph of girth ten into a forest and a matching,
\newblock {\em Stud. Appl. Math.} 124 (2010), 213--228.

\bibitem{BIS07}
O.V.Borodin, A.O. Ivanova, B.S. Stechkin,
\newblock Decomposing a planar graph into a forest and a subgraph of
restricted maximum degree,
\newblock {\em Sib. Elektron. Mat. Izv.}, 4 (2007), 296--299.

\bibitem{BIKS09}
O.V.~Borodin,  A.O.~Ivanova, A.V.~Kostochka, and N.N. Sheikh,
\newblock Planar graphs decomposable into a forest and a matching,
\newblock {\em Discrete Math.} 309 (2009), 277--279.

\bibitem{BIKS09a}
O.V.~Borodin, A.O. Ivanova, A.V. Kostochka and N.N.~Sheikh,
\newblock  Decompositions of quadrangle-free planar graphs,
\newblock {\em Discussiones Math. Graph Theory} 29 (2009), 87--99.

\bibitem{BK+02}
O.V.~Borodin, A.V.~Kostochka, N.N.~Sheikh, and G.~Yu.
\newblock Decomposing a planar graph with girth 9 into a forest and a matching.
\newblock {\em European J. Combinatorics} 29(5) (2008), 1235--1241.


\bibitem{G}
D. Gon\c{c}alves.
\newblock Covering planar graphs with forests, one having bounded maximum
degree.
\newblock {\em J. Combin. Theory (B)} 99 (2009), 314--322.

\bibitem{HH+02}
W.~He, X.~Hou, K.-W.~Lih, J.~Shao, W.~Wang, and X.~Zhu.
\newblock Edge-decompositions of planar graphs and their game coloring numbers.
\newblock {\em Journal of Graph Theory} 41 (2002), 307--317.

\bibitem{KKWWZ}
S.-J. Kim, A.V. Kostochka, H.~Wu, D.B.~West, and X.~Zhu.
\newblock Decomposition of sparse graphs into forests and a graph with
bounded degree.
\newblock {\em Journal of Graph Theory} 74 (2013), 369--391.

\bibitem{KL}
T. Kir\'aly and L. C. Lau,
\newblock Degree bounded forest covering,
\newblock {\em Integer programming and combinatorial optimization},
\newblock {\em Lect. Notes in Comput. Sci.} 6655 (Springer, 2011), 315--323.

\bibitem{Kleit}
D. J. Kleitman, Partitioning the edges of a girth 6 planar graph into those of
a forest and those of a set of disjoint paths and cycles, Manuscript, 2006.

\bibitem{Mont}
M. Montassier, A. P\^echer, A. Raspaud, D. B. West, and X.~Zhu.
\newblock Decomposition of sparse graphs, with application to game coloring
number.
\newblock {\em Discrete Math.} 310 (2010), 1520--1523.

\bibitem{MORZ}
M. Montassier, P. Ossona de Mendez, A. Raspaud, and X.~Zhu.
\newblock Decomposing a graph into forests.
\newblock {\em J. Combin. Theory (B)} 2012, 38--52.

\bibitem{NW}
C.~St.~J.~A.~Nash-Williams.
\newblock Decompositions of finite graphs into forests.
\newblock {\em J. London Math. Soc.} 39 (1964), 12.

\bibitem{Payan}
C. Payan,
\newblock Graphes equilibre et arboricit\'e rationnelle.
\newblock {\em Europ. J. Combin.}, 7 (1986), 263--270.

\bibitem{WZ}
Y. Wang and Q. Zhang,
\newblock Decomposing a planar graph with girth 8 into a forest and a matching,
\newblock {\em Discrete Math.} 311 (2011), 844--849.

\bibitem{Yang}
D. Yang, 
\newblock Decompose a graph into forests and a matching,
\newblock preprint.

\end{thebibliography}
}
\end{document}